\documentclass[journal]{IEEEtran}
\usepackage{xcolor}
\usepackage{amssymb}
\usepackage{amsmath}%for split
\usepackage{lineno}
\usepackage{cite}
\usepackage{graphicx}
\usepackage{hyperref}
\usepackage{subfigure}
\newtheorem{proof}{Proof}
\newtheorem{assumption}{Assumption}
\newtheorem{theorem}{Theorem}
\newtheorem{lemma}{Lemma}
\newtheorem{remark}{Remark}

\IEEEoverridecommandlockouts                              % This command is only
\title{\LARGE \bf
A Newton Tracking Algorithm with Exact Linear Convergence Rate for Decentralized Consensus Optimization
}

\author{Jiaojiao Zhang{$^*$}, Qing Ling{$^\dagger$}, and Anthony Man-Cho So{$^*$}% <-this % stops a space
\thanks{Qing Ling is supported in part by NSF China Grants 61573331 and 61973324, and Fundamental Research Funds for the Central Universities.}% <-this % stops a space
\thanks{Jiaojiao Zhang and Anthony Man-Cho So are with Department of Systems Engineering and Engineering Management, The Chinese University of Hong Kong.}%
\thanks{Qing Ling is with School of Data and Computer Science and Guangdong Province Key Laboratory of Computational Science, Sun Yat-Sen University.}%
}

\begin{document}

\maketitle
\thispagestyle{empty}
\pagestyle{empty}

%%%%%%%%%%%%%%%%%%%%%%%%%%%%%%%%%%%%%%%%%%%%%%%%%%%%%%%%%%%%%%%%%%%%%%%%%%%%%%%%
\begin{abstract}
This paper considers the decentralized consensus optimization problem defined over a network where each node holds a second-order differentiable local objective function. Our goal is to minimize the summation of local objective functions and find the exact optimal solution using only local computation and neighboring communication. We propose a novel Newton tracking algorithm, where each node updates its local variable along a local Newton direction modified with neighboring and historical information. We investigate the connections between the proposed Newton tracking algorithm and several existing methods, including gradient tracking and second-order algorithms. Under the strong convexity assumption, we prove that it converges to the exact optimal solution at a linear rate. Numerical experiments demonstrate the efficacy of Newton tracking and validate the theoretical findings.
\end{abstract}

%%%%%%%%%%%%%%%%%%%%%%%%%%%%%%%%%%%%%%%%%%%%%%%%%%%%%%%%%%%%%%%%%%%%%%%%%%%%%%%%
\section{INTRODUCTION}
In this paper, we focus on the decentralized consensus optimization problem defined over an undirected and connected network with $n$ nodes, in the form of
\begin{equation}\label{c1}
{x}^* =\arg\min_{x\in\mathbb{R}^{p}} ~ \sum_{i=1}^{n} f_{i}(x).
\end{equation}
Here, $f_i: \mathbb{R}^p\to \mathbb{R} \cup \{+\infty\}$ is a convex and second-order differentiable function privately owned by node $i$. Every agent aims to obtain an optimal solution ${x}^*$ to \eqref{c1} via local computation and communication with its neighbors. Decentralized consensus optimization problem in the form of \eqref{c1} arises in various applications, such as resource allocation \cite{pu2018push,liu2019decentralized,ta2019lora}, smart grid control \cite{emiliano2013decentralized, liu2019grid}, federate learning \cite{lalitha2018fully,li2019federated,zhao2019mobile}, decentralized machine learning \cite{lian2017fully,koppel2018decentralized,lee2019rl}, and so on.

%Decentralized optimization methods for solving problem \eqref{c1} have been developed
%in the literature. We identify the main algorithms as first-order and second-order algorithms.
%The decentralized consensus optimization problem \eqref{c1} appears in various fields, including network optimization \cite{pu2018push,scaman2018optimal}, optimization-based cooperative control \cite{dutta2017decentralized}, decentralized machine learning \cite{zhang2018distributed,koppel2018decentralized}, to name a few.

Decentralized consensus optimization methods have been extensively studied in the literature. Among the first-order methods, a popular algorithm is decentralized gradient descent (DGD) \cite{nedic2009distributed,yuan2016convergence,jakovetic2014fast}. However, DGD has to employ a diminishing step size to obtain an exact optimal solution. With a fixed step size, DGD converges fast but only to a neighborhood of an optimal solution \cite{jakovetic2014fast}. There are other first-order methods which use fixed step sizes but still converge to an exact optimal solution, including EXTRA \cite{shi2015extra}, decentralized ADMM \cite{shi2014linear,ling2015dlm}, exact diffusion \cite{yuan2018exact}, NIDS \cite{li2019decentralized}, and gradient tracking \cite{xu2015augmented,qu2017harnessing,xin2018linear,sun2019convergence,xin2020gradient}. Take gradient tracking as an example, each node maintains a local estimate of global gradient descent direction using neighboring and historical information, and uses it to correct the convergence error in DGD.

%Gradient tracking \cite{xu2015augmented,qu2017harnessing,xin2018linear,sun2019convergence,xin2020gradient} is a powerful tool for decentralized consensus optimization which is a combination of gradient descent and a gradient estimation scheme that utilizes history information to achieve fast and accurate estimation
%of the average gradient. Motivated by
%gradient tracking, this paper proposed a Newton tracking
%which can better use the second-order information.

Although the first-order algorithms enjoy the advantage of low iteration-wise computational complexity, second-order methods are still attractive due to their faster convergence speeds, and hence lower communication costs. Some works such as \cite{mokhtari2016network,bajovic2017newton,mansoori2019fast} penalize the implicit consensus constraints to the objective function. Hence, these penalized second-order algorithms can employ unconstrained optimization techniques, but only converge to a neighborhood of an optimal solution.  The penalty parameter trade-offs the convergence error and convergence speed. Primal-dual methods are effective to handle this accuracy-speed trade-off. This leads to the second-order methods operating in the primal-dual domain \cite{mokhtari2016dqm,mokhtari2016decentralized,eisen2019primal}, which achieve exact convergence with linear rates.
There exist other second-order methods with superlinear convergence rates under stricter conditions \cite{soori2019dave,zhang2020distributed}. For example, \cite{soori2019dave} proposes the distributed averaged quasi-Newton method for a master-slave network, but the initialization is required to be close enough to an optimal solution so as to guarantee locally superlinear convergence. The work of \cite{zhang2020distributed} runs a finite-time set-consensus inner loop at each iteration of the Polyak's adaptive Newton method, and  achieves globally superlinear convergence.

%Compared with first-order methods, second-order Newton methods are less studied in distributed optimization. Most of them \cite{mokhtari2016network,mansoori2019fast,tutunov2019distributed} transform the original problem \eqref{c1} into an inexact penalized problem to satisfy the consensus constraints. These penalized second-order algorithms only converge to an inexact solution. The penalty parameter works as a trade-off between the inexact error and convergence rate. Similar to the first-order methods, primal-dual methods are efficient to handle such inexactness-speed issue.  Some second-order algorithms based on primal-dual domain \cite{mokhtari2016dqm,mokhtari2016decentralized,eisen2019primal} are proposed to achieve exact convergence with linear rate.
%There are other second-order methods with superlinear convergence rates. For example,  \cite{soori2019dave} proposed a distributed averaged quasi-Newton methods on master-slave networks  and proved locally superlinear
%convergence when the starting point is close enough to an optimal solution.  Distributed adaptive Newton methods proposed in \cite{zhang2020distributed} achieves globally
%superlinear convergence by incorporating a finite time
%set-consensus method into Polyak's
%adaptive Newton method.

In this paper, we propose a novel second-order Newton tracking algorithm, in which each node updates its local variable along a local Newton direction modified with neighboring and historical information. As its name suggests, Newton tracking inherits the idea of gradient tracking, but improves its convergence speed through utilizing the second-order information. We investigate the connections between the proposed Newton tracking algorithm and several existing methods, including gradient tracking and second-order algorithms. Under the strong convexity assumption, we prove that it converges to the exact optimal solution at a linear rate. Numerical experiments demonstrate the efficacy of Newton tracking and validate the theoretical findings.

\textbf{Notations.} $\mathbf{I} \in \mathbb{R}^{np \times np}$, $I_n \in \mathbb{R}^{n \times n}$ and $I_p \in \mathbb{R}^{p \times p}$ denote identity matrices with different sizes. $\mathbf{0} \in \mathbb{R}^{np}$ and $0_p \in \mathbb{R}^p$ denote all-zero vectors with different sizes. $1_n \in \mathbb{R}^n$ is an all-one vector. $\lambda_{\max}(\cdot)$, $\lambda_{\min}(\cdot)$ and $\hat{\lambda}_{\min}(\cdot)$ denote the largest, smallest, and smallest nonzero eigenvalues of a matrix, respectively.

%performs a descent along a modified local Newton direction and then updates the Newton direction with a novel estimation scheme by communicating with its neighbors.
%This algorithm employs a fixed step size and, yet, converges to an exact solution. We investigate the connections between the Newton tracking algorithm and several existing methods, including gradient tracking and second-order algorithms.
%Under the strong convexity assumption, we prove that the
%proposed algorithm converges at a linear rate. We also provide
%some numerical experiments to demonstrate the efficacy of the
%proposed algorithm and to validate our theoretical findings.

%%%%%%%%%%%%%%%%%%%%%%%%%%%%%%%%%%%%%%%%%%%%%%%%%%%%%%%%%%%%%%%%%%%%%%%%%%%%%%%%
\section{Problem Formulation and Algorithm Development}

In this section, we rewrite the decentralized consensus optimization problem \eqref{c1} to an equivalent constrained form, and propose the Newton tracking algorithm to solve it.

\subsection{Problem Formulation}
Consider a bidirectionally connected network of $n$ nodes. Two nodes are neighbors if they are connected with an edge. Define $\mathcal{N}_i$ as the set of neighbors of node $i$ and let ${x}_{i} \in \mathbb{R}^{p}$ be the local copy of decision variable $x$ kept at node $i$. Since the network is bidirectionally connected, the optimization problem in \eqref{c1} is equivalent to
\begin{align}\label{d2}
\left\{{x}_{i}^{*}\right\}_{i=1}^{n}:= \underset{\left\{{x}_{i}\right\}_{i=1}^{n}}{{\arg\min}} & ~ \sum_{i=1}^{n} f_{i}\left({x}_{i}\right), \\\nonumber
\text { s.t. } & ~ {x}_{i}={x}_{j}, ~ \forall j \in \mathcal{N}_{i}, ~ \forall i.
\end{align}
Indeed, the constraint in \eqref{d2} enforces the consensus condition ${x}_{1}=\cdots={x}_{n}$ for any feasible solution of \eqref{d2}. When the consensus condition is
satisfied, the objective functions in \eqref{c1} and \eqref{d2} are equivalent, such that the optimal local variables ${x}_{i}^{*}$ of \eqref{d2} are all equal to the optimal argument $x^*$ of \eqref{c1}, i.e., ${x}_{1}^*=\cdots={x}_{n}^*=x^*$.

\subsection{Algorithm Development}
Let us introduce a nonnegative mixing matrix ${W} \in \mathbb{R}^{n \times n}$ whose $(i,j)$-th element $w_{i j} \geq 0$ represents the weight that node $i$ assigns to node $j$. The weight $w_{ij} = 0$ if and only if $j \notin $ $\mathcal{N}_{i} \cup\{i\}$. The mixing matrix $W$ is further required to satisfy the following assumption.
\begin{assumption}\label{assum1}
The mixing matrix $W$ is symmetric and doubly stochastic, i.e., $W=W^T$ and $W1_n=1_n$. The null space of $I_n-W$ is $\operatorname{span}\left(1_n\right)$.
\end{assumption}

When the underlying network is bidirectionally connected, the mixing matrix $W$ satisfying Assumption \autoref{assum1} can be generated using various techniques, such as those introduced in \cite{boyd2004fastest}. According to the Perron-Frobenius theorem \cite{pillai2005perron}, Assumption \autoref{assum1} means that the eigenvalues of $W$ lie in $(-1,1]$ and $W$ has a single eigenvalue at 1. 

At time $t$ of our proposed Newton tracking algorithm, each node $i$ keeps a local copy $x_i^t\in \mathbb{R}^p$ and a vector $u_i^t\in \mathbb{R}^p$ that estimates the negative Newton direction $u^t$, as $$ u_i^t \approx u^t \triangleq \left(\frac{1}{n}\sum_{i=1}^{n}\nabla^2 f_i(\bar{x}^{t}) \right)^{-1} \left(\frac{1}{n}\sum_{i=1}^{n}\nabla f_i(\bar{x}^{t})\right),$$ where $\bar{x}^t \triangleq \frac{1}{n} \sum_{i=1}^n x_i^t$ is the average of local copies. Each node $i$ updates $x_i^{t+1}$ from $x_i^t$ through descending along the direction $-u_i^t$ with unit step size. Since it is unaffordable to compute the exact Newton direction in a decentralized manner, we propose to estimate the Newton direction by a novel Newton tracking technique.

To be specific, the proposed Newton tracking algorithm starts with $x_i^{0} = 0_p$ and  $u_i^{0} = \left(\nabla^2 f_i(0_p)+\epsilon I_p\right)^{-1}\nabla f_i(0_p)$, then proceeds with
\begin{align}
\label{e3}x_i^{t+1}=&x_i^{t}-u_i^{t}, \\ \label{e3u}
u_i^{t+1}=& (\nabla^2 f_i(x_i^{t+1})+\epsilon I_p)^{-1}\\
\nonumber&\big[(\nabla^2 f_i(x_i^{t})+\epsilon I_p)u_i^{t}+ \nabla f_i(x_i^{t+1})-\nabla f_i(x_i^{t})\\
\nonumber& + 2\alpha(x_i^{t+1}-\sum_{j\in \mathcal{N}_i}w_{ij}x_j^{t+1})-\alpha(x_i^{t}-\sum_{j\in \mathcal{N}_i}w_{ij}x_j^{t})\big],
\end{align}
where $\epsilon>0$ and $\alpha>0$ are parameters. Comparing $-u_i^{t+1}$ with the true Newton direction, we have two observations. (i) The exact global Hessian $\frac{1}{n}\sum_{i=1}^{n}\nabla^2 f_i(\bar{x}^{t+1})$ is replaced by the regularized local Hessian $\nabla^2 f_i(x_i^{t+1})+\epsilon I_p$. The regularization parameter $\epsilon$ is necessary because the local Hessian $\nabla^2 f_i(x_i^{t+1})$ may be unreliable, especially in the beginning stage of the algorithm. (ii) The exact gradient $\frac{1}{n}\sum_{i=1}^{n}\nabla f_i(\bar{x}^{t})$ is replaced by three terms that are locally computable. The first term $(\nabla^2 f_i(x_i^{t})+\epsilon I_p)u_i^{t}$ involves the previous local Hessian and estimated Newton direction. The second term $\nabla f_i(x_i^{t+1})-\nabla f_i(x_i^{t})$ is the difference between the current and previous gradient directions. The third term $2\alpha(x_i^{t+1}-\sum_{j\in \mathcal{N}_i}w_{ij}x_j^{t+1})-\alpha(x_i^{t}-\sum_{j\in \mathcal{N}_i}w_{ij}x_j^{t})$ extrapolates the current and previous consensus errors.

Now we manipulate \eqref{e3u} to better illustrate the idea of Newton tracking. From \eqref{e3u} we have
\begin{align}
\label{e3u-new}
&(\nabla^2 f_i(x_i^{t+1})+\epsilon I_p) u_i^{t+1}\\
\nonumber=& (\nabla^2 f_i(x_i^{t})+\epsilon I_p)u_i^{t}+ \nabla f_i(x_i^{t+1})-\nabla f_i(x_i^{t})\\
\nonumber & + 2\alpha(x_i^{t+1}-\sum_{j\in \mathcal{N}_i}w_{ij}x_j^{t+1})-\alpha(x_i^{t}-\sum_{j\in \mathcal{N}_i}w_{ij}x_j^{t}).
\end{align}
Summing up \eqref{e3u-new} over $i=1,\ldots, n$ and invoking the double stochasticity of $W$,
we have
\begin{align} \label{e3u-new-new}
 &\sum_{i=1}^{n}\left(\nabla^2 f_i(x_i^{t+1})+\epsilon I_p\right)u_i^{t+1} \\
=&\sum_{i=1}^{n}\left(\nabla^2 f_i(x_i^{t})+\epsilon I_p\right)u_i^{t} + \sum_{i=1}^{n}\left(\nabla f_i(x_i^{t+1})-\nabla f_i(x_i^{t})\right). \nonumber
\end{align}
When the algorithm is initialized such that $\sum_{i=1}^{n}\nabla f_i(x_i^{0})$ $=\sum_{i=1}^{n}\left(\nabla^2 f_i(x_i^{0})+\epsilon I_p\right)u_i^{0}$, summing up \eqref{e3u-new-new} from time $0$ to time $t$ yields
$$\sum_{i=1}^{n}\left(\nabla^2 f_i(x_i^{t})+\epsilon I_p\right)u_i^{t}=\sum_{i=1}^{n}\nabla f_i(x_i^{t}).$$
In comparison, the global Newton direction $-u^t$ satisfies
$$\sum_{i=1}^{n}\nabla^2 f_i(\bar{x}^{t}) u^t = \sum_{i=1}^{n}\nabla f_i(\bar{x}^{t}).$$
When the local variable pairs $(x_i^t, u_i^t)$ are similar across the nodes, we observe that $x_i^t$ is close to $\bar{x}^t$ and $u_i^t$ tracks a regularized Newton direction.

%With particular note, the initialization condition $\sum_{i=1}^{n}\nabla f_i(x_i^{0})$ $=\sum_{i=1}^{n}\left(\nabla^2 f_i(x_i^{0})+\epsilon I_p\right)u_i^{0}$ introduced in this illustration is not necessary in analyzing the convergence property of Newton tracking algorithm.

The recursion \eqref{e3}-\eqref{e3u} can be written in a compact form. Define $\mathbf{x}\triangleq[{x}_{1} ; \ldots ; {x}_{n}]\in \mathbb{R}^{np}$ and $\mathbf{u}\triangleq[u_{1} ; \ldots ; u_{n}]\in \mathbb{R}^{np}$ as the
stacks of local variables. Define the aggregate function $ f: \mathbb{R}^{n p} \rightarrow \mathbb{R}$ as $f(\mathbf{x})=f({x}_{1}, \cdots, {x}_{n})=\sum_{i=1}^{n} f_{i}({x}_{i})$ that sums up all the local functions $f_{i}\left(x_{i}\right)$. The gradient of  $f(\mathbf{x})$ is $\nabla f(\mathbf{x})=[\nabla f_1({x}_{1}) ; \ldots ; \nabla f_n({x}_{n})]\in \mathbb{R}^{np}$. The Hessian of $f(\mathbf{x})$, denoted by $\nabla^2 f(\mathbf{x}) \in \mathbb{R}^{np\times
np}$, is a block diagonal matrix whose $i$-th diagonal block is $\nabla^2 f_i(x)$. Define $\mathbf{H} \triangleq \nabla^2 f(\mathbf{x})+\epsilon \mathbf{I} \in\mathbb{R}^{np\times
np}$ and $\mathbf{W}\triangleq{W} \otimes I_{p} \in \mathbb{R}^{np\times np} $ as the Kronecker product of the weight matrix $W$ and
the identity matrix $I_p$. The recursion \eqref{e3}-\eqref{e3u} can be written as
\begin{align}
\label{e4}\mathbf{x}^{t+1}=&\mathbf{x}^t-\mathbf{u}^t,\\\label{e4u}
\mathbf{u}^{t+1}=&({\mathbf{H}^{t+1}})^{-1}\big[\mathbf{H}^t \mathbf{u}^t  + \nabla f\left(\mathbf{x}^{t+1}\right)-\nabla f\left(\mathbf{x}^t\right) \\\nonumber
& +\alpha(\mathbf{I} -\mathbf{W})(2\mathbf{x}^{t+1}-\mathbf{x}^t) \big].
\end{align}
The algorithm is initialized as $\mathbf{x}^0 = \mathbf{0}$ and $\mathbf{u}^0 = (\nabla^2 f(\mathbf{0})+\epsilon \mathbf{I})^{-1} \nabla f(\mathbf{0})$.

\section{Connections with Existing Approaches}
This section investigates the connections of the proposed Newton tracking algorithm with several existing approaches, such as gradient tracking and primal-dual methods.

\subsection{Connection with Gradient Tracking}
The recursion of gradient tracking \cite{qu2017harnessing} is given by
\begin{align}
&\mathbf{x}^{t+1}=\mathbf{W} \mathbf{x}^{t}-\alpha \mathbf{y}^{t}, \label{e5} \\
&\mathbf{y}^{t+1}=\mathbf{W} \mathbf{y}^{t}+\nabla {f}(\mathbf{x}^{t+1})-\nabla {f}(\mathbf{x}^{t}), \label{e5y}
\end{align}
where $\mathbf{x},\mathbf{y}\in \mathbb{R}^{np}$. To see the connection between gradient tracking and Newton tracking, we rewrite \eqref{e5}-\eqref{e5y}. First, write $\mathbf{x}^{t+1}=\mathbf{W} \mathbf{x}^{t}-\alpha \mathbf{y}^{t}$ as $\mathbf{x}^{t+1}=\mathbf{x}^{t} - [(\mathbf{I} -\mathbf{W}) \mathbf{x}^{t}+\alpha \mathbf{y}^{t}]$. Then, define $ \mathbf{r}^t= (\mathbf{I}-\mathbf{W}) \mathbf{x}^{t}+\alpha \mathbf{y}^{t} \in\mathbb{R}^{np}$. Replacing $\mathbf{y}$ with $\mathbf{r}$ shows that \eqref{e5}-\eqref{e5y} are equivalent to
\begin{align}
\mathbf{x}^{t+1}=& \mathbf{x}^{t}-\mathbf{r}^{t},\label{e6}\\
\mathbf{r}^{t+1}=&\mathbf{W}\mathbf{r}^{t}+\alpha[\nabla {f}(\mathbf{x}^{t+1})-\nabla {f}(\mathbf{x}^{t})]  \label{e6r} \\\nonumber& +(\mathbf{I} -\mathbf{W})(\mathbf{x}^{t+1}-\mathbf{W}\mathbf{x}^{t}).
\end{align}

Similar to the update of $\mathbf{u}^{t+1}$ in \eqref{e4u}, the update of $\mathbf{r}^{t+1}$ in \eqref{e6r} also involves three parts: the previous direction $\mathbf{r}^{t}$, the difference between current and previous gradient directions $\alpha[\nabla {f}(\mathbf{x}^{t+1})-\nabla {f}(\mathbf{x}^{t})]$, and the combination of current and previous consensus errors $(\mathbf{I} -\mathbf{W})(\mathbf{x}^{t+1}-\mathbf{W}\mathbf{x}^{t})$. The major difference between $\mathbf{u}^{t+1}$ and $\mathbf{r}^{t+1}$ is that the former utilizes the current and previous Hessians, which help improve the convergence speed, especially when the local objective functions are ill-conditioned.

\subsection{Connection with Primal-dual Algorithms}

The proposed Newton tracking algorithm has a primal-dual interpretation. Note that the null space of $I_n-W$ is $\operatorname{span}\left(1_n \right)$, so is the null space of its square root $(I_n-W)^{\frac{1}{2}}$. Because $(\mathbf{I}-\mathbf{W})^{\frac{1}{2}} = (I_n-W)^{\frac{1}{2}} \otimes I_p$, $(\mathbf{I}-\mathbf{W})^{\frac{1}{2}}\mathbf{x}=\mathbf{0}$ if and only if $x_1=\cdots =x_n$.  The optimization problem \eqref{d2} is equivalent to
\begin{align}\label{c5}
\mathbf{x}^*\triangleq \arg\min_\mathbf{x} & ~ f(\mathbf{x}),\\\nonumber
\text{s.t.} & ~ (\mathbf{I} -\mathbf{W})^{\frac{1}{2}}\mathbf{x}=\mathbf{0}.
\end{align}

The augmented Lagrangian $L(\mathbf{x},\mathbf{v})$ of \eqref{c5} is
\begin{align}\label{c6}
\hspace{-1em} L(\mathbf{x},\mathbf{v})=f(\mathbf{x})+\langle \mathbf{v},(\mathbf{I} -\mathbf{W})^{\frac{1}{2}}\mathbf{x} \rangle+\frac{\alpha}{2} \mathbf{x}^T(\mathbf{I} -\mathbf{W})\mathbf{x},
\end{align}
where $\mathbf{v}\in \mathbb{R}^{np}$ is the dual variable. Therefore, the augmented Lagrangian method to solve \eqref{c5} is given by
\begin{align}
\mathbf{x}^{t+1}&=\arg\min_{\mathbf{x}} ~ L(\mathbf{x},\mathbf{v}^k), \label{fajdisfjoas} \\
\mathbf{v}^{t+1}&=\mathbf{v}^t+\alpha(\mathbf{I} -\mathbf{W})^{\frac{1}{2}} \mathbf{x}^{t+1}. \label{6}
\end{align}

However, solving \eqref{fajdisfjoas} is nontrivial. First, $f(\mathbf{x})$ is a general objective function such that \eqref{fajdisfjoas} does not have a closed-form solution. Second, even if $f(\mathbf{x})$ is quadratic, the topology-dependent quadratic term $\frac{\alpha}{2} \mathbf{x}^T(\mathbf{I} -\mathbf{W})\mathbf{x}$ makes the closed-form solution not implementable in a decentralized manner. Motivated by these observations, we quadratically approximate $f(\mathbf{x})$ and linearly approximate $\frac{\alpha}{2} \mathbf{x}^T(\mathbf{I} -\mathbf{W})\mathbf{x}$ both around $\mathbf{x}^t$, and then add a proximal term $\frac{\epsilon}{2}\|\mathbf{x}-\mathbf{x}^t\|^2$ to the objective function of \eqref{fajdisfjoas}. This way, the update of $\mathbf{x}^{t+1}$ is given by the solution to
\begin{align*}%\label{c8}
\hspace{-1em}\min_{\mathbf{x}}& ~ \left\langle \triangledown f(\mathbf{x}^t)\!+\!(\mathbf{I} -\mathbf{W})^{\frac{1}{2}}\mathbf{v}^t +\alpha(\mathbf{I} -\mathbf{W})\mathbf{x}^t, \mathbf{x}-\mathbf{x}^t \!\right\rangle\\\nonumber & ~ + \frac{1}{2}(\mathbf{x}-\mathbf{x}^t)^T\triangledown^2 f(\mathbf{x}^t)(\mathbf{x}-\mathbf{x}^t)+\frac{\epsilon}{2}\|\mathbf{x}-\mathbf{x}^t\|^2,
\end{align*}
which is
\begin{align}\label{b5}
&\mathbf{x}^{t+1} = \mathbf{x}^t \\\nonumber
&\hspace{1em}-\left( \mathbf{H}^t \right)^{-1} \left[\nabla f\left(\mathbf{x}^t\right)+(\mathbf{I} -\mathbf{W})^{\frac{1}{2}} \mathbf{v}^t+\alpha(\mathbf{I}-\mathbf{W}) \mathbf{x}^t\right].
\end{align}

Next, we show that \eqref{b5} and \eqref{6} initialized by $\mathbf{x}^0 = \mathbf{0}$ and $\mathbf{v}^0 = \mathbf{0}$ are equivalent to \eqref{e4}-\eqref{e4u} initialized by $\mathbf{x}^0 = \mathbf{0}$ and $\mathbf{u}^0 = (\nabla^2 f(\mathbf{0})+\epsilon \mathbf{I})^{-1} \nabla f(\mathbf{0})$. By \eqref{b5}, the two recursions have the same $\mathbf{x}^1 = -(\nabla^2 f(\mathbf{0})+\epsilon \mathbf{I})^{-1} \nabla f(\mathbf{0})$. Also by \eqref{b5}, we have
\begin{align} \label{b7}
&\mathbf{H}^t\mathbf{x}^{t+1} = \mathbf{H}^t\mathbf{x}^t \\
&\hspace{1em}-\left[\nabla f\left(\mathbf{x}^t\right)+ (\mathbf{I} -\mathbf{W})^{\frac{1}{2}} \mathbf{v}^t+\alpha(\mathbf{I} -\mathbf{W}) \mathbf{x}^t\right], \nonumber \\ \label{b8}
&\mathbf{H}^{t+1}\mathbf{x}^{t+2} =\mathbf{H}^{t+1}\mathbf{x}^{t+1} \\
&\hspace{1em}-\left[\nabla f\left(\mathbf{x}^{t+1}\right)+ (\mathbf{I} -\mathbf{W})^{\frac{1}{2}} \mathbf{v}^{t+1}+\alpha(\mathbf{I} -\mathbf{W}) \mathbf{x}^{t+1}\right], \nonumber \end{align}
Subtracting \eqref{b7} from \eqref{b8} and substituting the dual update \eqref{6} to eliminate the terms $(\mathbf{I} -\mathbf{W})^{\frac{1}{2}} \mathbf{v}^t$ and $(\mathbf{I} -\mathbf{W})^{\frac{1}{2}} \mathbf{v}^{t+1}$, we have
\begin{align*}%\label{b9}
&\mathbf{H}^{t+1}\mathbf{x}^{t+2}=\left[\mathbf{H}^t+\mathbf{H}^{t+1}-2\alpha(\mathbf{I} -\mathbf{W})\right]\mathbf{x}^{t+1}\\\nonumber
&-\left[\mathbf{H}^t-\alpha(\mathbf{I} -\mathbf{W})\right]\mathbf{x}^t
-\left[\nabla f\left(\mathbf{x}^{t+1}\right)-\nabla f\left(\mathbf{x}^t\right)\right],
\end{align*}
which is equivalent to
\begin{align}\label{b10}
&\mathbf{H}^{t+1}\mathbf{x}^{t+2}-\left[\mathbf{H}^{t+1}-\alpha(\mathbf{I} -\mathbf{W})\right]\mathbf{x}^{t+1}\\\nonumber
=&\mathbf{H}^t\mathbf{x}^{t+1}-\left[\mathbf{H}^t-\alpha(\mathbf{I} -\mathbf{W})\right]\mathbf{x}^t-\alpha(\mathbf{I} -\mathbf{W})\mathbf{x}^{t+1}\\\nonumber
&-\left[\nabla f\left(\mathbf{x}^{t+1}\right)-\nabla f\left(\mathbf{x}^t\right)\right].
\end{align}

Defining $\mathbf{s}^t\triangleq \mathbf{H}^t\mathbf{x}^{t+1}-\left[\mathbf{H}^t-\alpha(\mathbf{I} -\mathbf{W})\right]\mathbf{x}^t$, we rewrite \eqref{b10} as
\begin{align}\label{b11}
\hspace{-1em}\mathbf{s}^{t+1}=&\mathbf{s}^t-\alpha(\mathbf{I}-\mathbf{W})\mathbf{x}^{t+1}-\left[\nabla f\left(\mathbf{x}^{t+1}\right)-\nabla f\left(\mathbf{x}^t\right)\right].
\end{align}
From the definition of $\mathbf{s}^t$, it holds
\begin{align}\label{b12}
\mathbf{x}^{t+1}&=\mathbf{x}^t-({\mathbf{H}^t})^{-1}\left[\alpha(\mathbf{I} -\mathbf{W})\mathbf{x}^t-\mathbf{s}^t\right].
\end{align}
Further defining $\mathbf{q}^t\triangleq \alpha(\mathbf{I} -\mathbf{W})\mathbf{x}^t-\mathbf{s}^t = \mathbf{H}^t (\mathbf{x}^t-\mathbf{x}^{t+1})$, we rewrite \eqref{b12} and \eqref{b11} as
\begin{align}
\label{b13}\mathbf{x}^{t+1}=&\mathbf{x}^t-({\mathbf{H}^t})^{-1}\mathbf{q}^t,
\\\label{b13q}
\mathbf{q}^{t+1}=&\mathbf{q}^t+ \nabla f\left(\mathbf{x}^{t+1}\right)-\nabla f\left(\mathbf{x}^t\right) \\\nonumber
& +\alpha(\mathbf{I} -\mathbf{W})(2\mathbf{x}^{t+1}-\mathbf{x}^t).
\end{align}
Observe that \eqref{b13}-\eqref{b13q} are equivalent to \eqref{e4}-\eqref{e4u} in the sense of $\mathbf{u}^t = ({\mathbf{H}^t})^{-1}\mathbf{q}^t$.

\begin{remark}
There is an existing primal-dual second-order algorithm called ESOM that also quadratically approximates the augmented Lagrangian when solving \eqref{fajdisfjoas} \cite{mokhtari2016decentralized}. However, unlike the proposed Newton tracking algorithm, ESOM does not linearize the topology-dependent quadratic term $\frac{\alpha}{2} \mathbf{x}^T(\mathbf{I} -\mathbf{W})\mathbf{x}$, which, as we have indicated earlier, makes the closed-form solution not implementable in a decentralized manner. In fact, the primal update of ESOM is given by
\begin{align}\label{d5}
\mathbf{x}^{t+1}=&\mathbf{x}^t-\left(\triangledown^2 f(\mathbf{x})+\alpha (\mathbf{I} -\mathbf{W})+\epsilon \mathbf{I}\right)^{-1}\\&\big[\nabla f\left(\mathbf{x}^t\right)+(\mathbf{I} -\mathbf{W})^{\frac{1}{2}} \mathbf{v}^t+\alpha(\mathbf{I} -\mathbf{W}) \mathbf{x}^t\big]. \nonumber
\end{align}
In \eqref{d5}, computing the inverse of $\triangledown^2 f(\mathbf{x})+\alpha (\mathbf{I} -\mathbf{W})+\epsilon \mathbf{I}$ requires multiple rounds of communication and computation. Therefore, ESOM introduces an inner loop to approximately solve \eqref{d5}, which leads to extra communication and computation costs \cite{mokhtari2016decentralized}.
\end{remark}

%%%%%%%%%%%%%%%%%%%%%%%%%%%%%%%%%%%%%%%%%%%%%%%%%%%%%%%%%%%%%%%%%%%%%%%%%%%%%%%%
\section{convergence analysis}

Since the Newton tracking recursion \eqref{e4}-\eqref{e4u} is equivalent to the primal-dual iteration in \eqref{b5} and \eqref{6}, once we show that the primal-dual iteration in \eqref{b5} and \eqref{6} exhibits a linear convergence rate, then so does the Newton tracking recursion \eqref{e4}-\eqref{e4u}. In the analysis, we need the following assumption.

\begin{assumption}\label{assum2}
The local objective functions $f_i(x_i)$ are twice differentiable. The eigenvalues of Hessians $\nabla^2 f_i(x_i)$ are bounded by positive constants $\mu_f, L_f \in (0,\infty)$, i.e.
\begin{align}\label{b1}
\mu_f I_p \preceq \nabla^{2} f_{i}\left(x_i \right) \preceq L_f I_p,
\end{align}
for all $x_i  \in \mathbb{R}^{p}$ and $i=1, \ldots, n$.
\end{assumption}

The lower bound in \eqref{b1} implies that the local objective functions $f_i(x)$ are strongly convex with constant $\mu_f>0$. The upper bound implies that the local gradients $\nabla f_i(x)$ are Lipschitz continuous with constant $L_f$. Note that the aggregate objective function $\nabla^{2} f(\mathbf{x})$ is a block diagonal matrix whose $i$-th diagonal block is $\nabla^2 f_i(x)$. Therefore, the bounds on the eigenvalues of Hessians $\nabla^2 f_i(x)$  in \eqref{b1} also hold for the aggregate Hessian, i.e.
\begin{align*}%\label{b2}
\mu_f \mathbf{I} \preceq \nabla^{2} f(\mathbf{x}) \preceq L_f \mathbf{I},
\end{align*}
for all $\mathbf{x} \in \mathbb{R}^{n p}$. Thus, the aggregate objective function $f(\mathbf{x})$ is also strongly convex with constant $\mu_f$ and its gradients $\nabla f(\mathbf{x})$ are Lipschitz continuous with constant $L_f$.

%\begin{assumption}\label{assum3}
%The aggregate Hessian $\nabla^{2} f(\mathbf{x})$ is Lipschitz continuous with constant $L$, i.e.,
%\begin{equation}\label{b3}
%\left\|\nabla^{2} f(\mathbf{x})-\nabla^{2} f(\tilde{\mathbf{x}})\right\| \leq L\|\mathbf{x}-\tilde{\mathbf{x}}\|, ~ \forall \mathbf{x}, \tilde{\mathbf{x}} \in \mathbb{R}^{np}.
%\end{equation}
%\end{assumption}
%
%\vspace{0.5em}
%
%The condition imposed by Assumption \autoref{assum3} is common in
%the analysis of second-order methods.

Our analysis involves the optimal primal-dual pair $(\mathbf{x}^*,\mathbf{v}^*)$ of \eqref{c5}. According to the KKT condition of \eqref{c5}, we have
\begin{align}\label{12}
\triangledown f(\mathbf{x}^*)+(\mathbf{I}-\mathbf{W})^{\frac{1}{2}}\mathbf{v}^*=\mathbf{0},\\ \label{13}
(\mathbf{I}-\mathbf{W})^{\frac{1}{2}}\mathbf{x}^*=0 \quad \text{or} \quad (\mathbf{I}-\mathbf{W})\mathbf{x}^*=\mathbf{0}.
\end{align}

\begin{lemma}\label{lemma1}
Consider the equivalent Newton tracking iteration in \eqref{b5} and \eqref{6}. The primal-dual iterate satisfies
\begin{align}\label{7}
\nabla f\left(\mathbf{x}^{t+1}\right) &-\nabla f\left(\mathbf{x}^*\right)+(\mathbf{I}-\mathbf{W})^{\frac{1}{2}}\left(\mathbf{v}^{t+1}-\mathbf{v}^*\right)\\\nonumber
&+\epsilon\left(\mathbf{x}^{t+1}-\mathbf{x}^t\right)+\mathbf{e}^t=\mathbf{0},
\end{align}
where $\mathbf{e}^t$ is defined as
\begin{align}\label{8}
\mathbf{e}^t\triangleq  & \nabla f\left(\mathbf{x}^t\right)-\nabla f\left(\mathbf{x}^{t+1}\right)+\nabla^{2} f\left(\mathbf{x}^t\right)\left(\mathbf{x}^{t+1}-\mathbf{x}^t\right)\\\nonumber
&-\alpha(\mathbf{I}-\mathbf{W})(\mathbf{x}^{t+1}-\mathbf{x}^t).
\end{align}
\end{lemma}
The result in Lemma \autoref{lemma1} shows the relationship of the primal-dual pairs $(\mathbf{x}^t,\mathbf{v}^t)$ and $(\mathbf{x}^{t+1},\mathbf{v}^{t+1})$ with the optimal primal-dual pair $(\mathbf{x}^*,\mathbf{v}^*)$. 
The arguments used in the
proof of Lemma \autoref{lemma1} are similar to ones used in \cite{mokhtari2016decentralized}. 

\begin{proof}
By the definition of $\mathbf{e}^t$, \eqref{b5} can be rewritten as
\begin{align}\label{10}
&{\nabla f\left(\mathbf{x}^{t+1}\right)+(\mathbf{I}-\mathbf{W})^{\frac{1}{2}} \mathbf{v}^t}
+\alpha(\mathbf{I}-{\mathbf{W}}) \mathbf{x}^{t+1}\\\nonumber
& \quad +\epsilon\left(\mathbf{x}^{t+1}-\mathbf{x}^t\right)+\mathbf{e}^t=\mathbf{0}.
\end{align}
%The Lagrange is
%\begin{align}
%L(\mathbf{x},\mathbf{v})=f(\mathbf{x})+\left\langle \mathbf{v},(\mathbf{I}-\mathbf{W})^{\frac{1}{2}}\mathbf{x}\right\rangle
%\end{align}
Combining \eqref{12} and \eqref{13} with \eqref{10}, we have
\begin{align}\label{14}
&{\nabla f\left(\mathbf{x}^{t+1}\right)-\nabla f\left(\mathbf{x}^*\right)+(\mathbf{I}-\mathbf{W})^{\frac{1}{2}}\left(\mathbf{v}^t-\mathbf{v}^*\right)} \\\nonumber
&\quad +\alpha(\mathbf{I}-\mathbf{W})\left(\mathbf{x}^{t+1}-\mathbf{x}^*\right)+\epsilon\left(\mathbf{x}^{t+1}-\mathbf{x}^t\right)+\mathbf{e}^t=\mathbf{0}.
\end{align}
Observe that $\mathbf{v}^t$ in \eqref{14} can be further replaced by $\mathbf{v}^{t+1}$. To be specific, substituting \eqref{13} into \eqref{6} and then regrouping terms, we know that $\mathbf{v}^t$ can be represented as
\begin{align}\label{15}
\mathbf{v}^t=\mathbf{v}^{t+1}-\alpha(\mathbf{I}-\mathbf{W})^{\frac{1}{2}}\left(\mathbf{x}^{t+1}-\mathbf{x}^*\right).
\end{align}
Substituting \eqref{15} into \eqref{14} yields \eqref{7}.
\end{proof}

Observe that the term $\mathbf{e}^t$ can be interpreted as the error introduced by approximation at time $t$, which motivates us to find an upper bound for $\|\mathbf{e}^t\|$. In the following lemma, we provide an upper bound for $\left\|\mathbf{e}^t\right\|$ in terms of $\left\|\mathbf{x}^{t+1}-\mathbf{x}^t\right\|$.

\begin{lemma}\label{lemma2}
Consider the equivalent Newton tracking iteration in \eqref{b5} and \eqref{6}, and recall the definition of the error vector $\mathbf{e}^t$ in \eqref{8}. If Assumption	\autoref{assum2} holds, then $\|\mathbf{e}^t\|$ is bounded by
\begin{align}\label{c28}
\left\|\mathbf{e}^t\right\| \leq \kappa \left\|\mathbf{x}^{t+1}-\mathbf{x}^t\right\|.
\end{align}
where $\kappa \triangleq 2L_f+\alpha \lambda_{\max}(\mathbf{I}-\mathbf{W})$.
\end{lemma}

\begin{proof}
By the triangle inequality, $\|\mathbf{e}^t\|$ is bounded by
\begin{align}\label{25}
& \|\mathbf{e}^t\| \leq \| \nabla f\left(\mathbf{x}^t\right) -\nabla f\left(\mathbf{x}^{t+1}\right) \| \\
& +\|\nabla^{2} f\left(\mathbf{x}^t\right)\left(\mathbf{x}^{t+1}-\mathbf{x}^t\right)\| + \| \alpha(\mathbf{I}-\mathbf{W})(\mathbf{x}^{t+1}-\mathbf{x}^t) \|. \nonumber
\end{align}
By Assumption \autoref{assum2}, $\| \nabla f\left(\mathbf{x}^t\right) -\nabla f\left(\mathbf{x}^{t+1}\right) \| \leq L_f \|\mathbf{x}^{t+1}-\mathbf{x}^t\|$. As the largest eigenvalue of $\nabla^{2} f\left(\mathbf{x}^t\right)$ and $\mathbf{I}-\mathbf{W}$ are $L_f$ and $\lambda_{\max}(\mathbf{I}-\mathbf{W})$, respectively, we know $\|\nabla^{2} f\left(\mathbf{x}^t\right) \left(\mathbf{x}^{t+1}\right.$ $\left.-\mathbf{x}^t\right)\| \leq L_f \|\mathbf{x}^{t+1}-\mathbf{x}^t\|$ and $\| \alpha(\mathbf{I}-\mathbf{W})(\mathbf{x}^{t+1}-\mathbf{x}^t) \| \leq \lambda_{\max}(\mathbf{I}-\mathbf{W}) \|\mathbf{x}^{t+1}-\mathbf{x}^t\|$. Substituting these inequalities into \eqref{25} completes the proof.
\end{proof}

The result in \eqref{c28} demonstrates that the error $\mathbf{e}^t$ introduced by approximation becomes zero as the sequence of iterates $\mathbf{x}^t$ approaches the optimal solution $\mathbf{x}^*$, which will be shown in \autoref{theom1}.

Given the preliminary results in Lemmas \autoref{lemma1} and \autoref{lemma2}, we are ready to establish the linear convergence of the proposed Newton tracking method. To do so, we define vectors $\mathbf{\zeta}, \mathbf{\zeta}^* \in \mathbb{R}^{2np}$ and a matrix $\mathbf{G}\in \mathbb{R}^{np\times np}$ as
$$\mathbf{\zeta}^t=\left[\begin{array}{l}
{\mathbf{x}^t} \\
{\mathbf{v}^t}
\end{array}\right], ~ \mathbf{\zeta}^*=\left[\begin{array}{l}
{\mathbf{x}^*} \\
{\mathbf{v}^*}
\end{array}\right], ~ \mathbf{G}=\left[\begin{array}{cc}
\mathbf{Q} & \mathbf{0} \\
\mathbf{0} & {\frac{1}{\alpha} \mathbf{I}}
\end{array}\right],$$
where $\mathbf{Q}\triangleq \epsilon \mathbf{I}-\alpha (\mathbf{I}-\mathbf{W})$. Note that $\mathbf{Q}$ is positive definite when $\epsilon - \alpha \lambda_{\max}(\mathbf{I}-\mathbf{W}) > 0$. In the following theorem, we show that the sequence $\|\mathbf{\zeta}^t-\mathbf{\zeta}^*\|_{\mathbf{G}}$  converges to zero at a linear rate.
\begin{theorem}\label{theom1}
Consider the equivalent Newton tracking iteration in \eqref{b5} and \eqref{6}. Suppose that the parameters $\epsilon$ and $\alpha$ satisfy $\lambda_{\min}(\mathbf{Q}) = \epsilon-\alpha \lambda_{\max}(\mathbf{I}-\mathbf{W}) > \frac{4L_f^2}{\mu_f}$. Then, the sequence of $\left\|\mathbf{\zeta}^t-\mathbf{\zeta}^*\right\|_{\mathbf{G}}^{2}$ satisfies
\begin{align}\label{26}
\left\|\mathbf{\zeta}^{t+1}-\mathbf{\zeta}^*\right\|_{\mathbf{G}}^{2} \leq \frac{1}{1+{\delta}^{\prime}}\left\|\mathbf{\zeta}^t-\mathbf{\zeta}^*\right\|_{\mathbf{G}}^{2},
\end{align}
where
\begin{align}\label{xx}
{\delta}^{\prime} = & \min\left\{\frac{\mu_f\delta}{(1+\delta)\left[ \epsilon +\frac{\beta\phi L_f^2}{{\alpha \hat{\lambda}_{\min}(\mathbf{I}-\mathbf{W})}}\right]},\right. \nonumber \\& \left.\frac{\alpha \delta^2(\epsilon-\alpha \lambda_{\max}(\mathbf{I}-\mathbf{W})){ \hat{\lambda}_{\min}(\mathbf{I}-\mathbf{W})}}{\frac{\beta\epsilon^2}{(\beta-1)}+\frac{\beta\phi(2L_f+\alpha \lambda_{\max}(\mathbf{I}-\mathbf{W}))^2}{(\phi-1)}}\right\}.
\end{align}
Therein, $\beta> 1$ and $ \phi> 1$ are arbitrary constants, and
\begin{equation}
\delta \triangleq 1-\frac{4L_f^2}{\mu_f {\lambda_{\min }({\mathbf{Q}})} } = 1-\frac{4L_f^2}{\mu_f (\epsilon-\alpha \lambda_{\max}(\mathbf{I}-\mathbf{W}))}>0. \nonumber
\end{equation}
\end{theorem}
\vspace{1em}

\begin{proof}
\textbf{Step 1.} By reorganizing \eqref{b5}, we get
\begin{align*}%\label{a38}
&\epsilon(\mathbf{x}^t-\mathbf{x}^{t+1}) +\triangledown^2 f(\mathbf{x}^t)\left(\mathbf{x}^{t}-\mathbf{x}^{t+1}\right)\\\nonumber
&-\left[ \nabla f\left(\mathbf{x}^t\right)+(\mathbf{I}-\mathbf{W})^{\frac{1}{2}} \mathbf{v}^t+\alpha(\mathbf{I}-\mathbf{W}) \mathbf{x}^t \right]=\mathbf{0}.
\end{align*}
Thus, it holds
\begin{align}\label{a39}
&\left\langle \mathbf{x}^*-\mathbf{x}^{t+1}, \epsilon(\mathbf{x}^t-\mathbf{x}^{t+1}) + \triangledown^2 f(\mathbf{x}^t)\left(\mathbf{x}^t-\mathbf{x}^{t+1}\right)\right.\\&\left.
-\left[ \nabla f\left(\mathbf{x}^t\right)+(\mathbf{I}-\mathbf{W})^{\frac{1}{2}} \mathbf{v}^t+\alpha(\mathbf{I}-\mathbf{W}) \mathbf{x}^t \right]\right\rangle = {0}.\nonumber
\end{align}
Substituting the dual update $\mathbf{v}^t=\mathbf{v}^{t+1}-\alpha(\mathbf{I}-\mathbf{W})^{\frac{1}{2}} \mathbf{x}^{t+1}$ and regrouping the terms, we can rewrite \eqref{a39} to
%\begin{align}\label{a40}
%& 0=\left\langle \mathbf{x}^*-\mathbf{x}^{t+1}, \epsilon(\mathbf{x}^t-\mathbf{x}^{t+1}) +\triangledown^2 f(\mathbf{x}^t)\left(\mathbf{x}^t-\mathbf{x}^{t+1}\right)\right.\\\nonumber&\left.
%-\big[ \nabla f\left(\mathbf{x}^t\right)+(\mathbf{I}-\mathbf{W})^{\frac{1}{2}} \mathbf{v}^{t+1}-\alpha(\mathbf{I}-\mathbf{W}) \mathbf{x}^{t+1}+\alpha(\mathbf{I}-\mathbf{W}) \mathbf{x}^t \big]\right\rangle
%\end{align}
\begin{align}\label{a41}
&\left\langle \mathbf{x}^*-\mathbf{x}^{t+1}, \underbrace{\left( \epsilon \mathbf{I}  -\alpha(\mathbf{I}-\mathbf{W}) \right)}_{\triangleq \mathbf{Q}}(\mathbf{x}^t-\mathbf{x}^{t+1})\right\rangle\\
&-\left\langle \mathbf{x}^*-\mathbf{x}^{t+1},
\nabla f\left(\mathbf{x}^t\right) \right\rangle -\left\langle \mathbf{x}^*-\mathbf{x}^{t+1},
(\mathbf{I}-\mathbf{W})^{\frac{1}{2}} \mathbf{v}^{t+1}\right\rangle \nonumber\\
&+\left\langle \mathbf{x}^*-\mathbf{x}^{t+1}, \triangledown^2 f(\mathbf{x}^t) (\mathbf{x}^t-\mathbf{x}^{t+1}) \right\rangle
= {0}.\nonumber
\end{align}
For the first term at the left-hand side of \eqref{a41}, we have
\begin{align}\label{a42}
 &\left\langle \mathbf{x}^*-\mathbf{x}^{t+1}, \mathbf{Q} (\mathbf{x}^t-\mathbf{x}^{t+1})\right\rangle\\
=&\frac{1}{2}\big(\|\mathbf{x}^*-\mathbf{x}^{t+1}\|_{\mathbf{Q}}^2+\|\mathbf{x}^t-\mathbf{x}^{t+1}\|_{\mathbf{Q}}^2-\|\mathbf{x}^*-\mathbf{x}^t\|_{\mathbf{Q}}^2\big). \nonumber \end{align}
For the second term at the left-hand side of \eqref{a41}, according to the $\mu_f$-strong convexity of $f$, we have
\begin{align}\label{a43}
&\left\langle \mathbf{x}^*-\mathbf{x}^{t+1}, \nabla f\left(\mathbf{x}^t\right)\right\rangle \\
=&\left\langle \mathbf{x}^*-\mathbf{x}^{t+1}, \nabla f\left(\mathbf{x}^{t+1}\right)\right\rangle \nonumber \\
&+\left\langle \mathbf{x}^*-\mathbf{x}^{t+1}, \nabla f\left(\mathbf{x}^t\right)-\nabla f\left(\mathbf{x}^{t+1}\right)\right\rangle \nonumber\\
\leq & f(\mathbf{x}^*)-f\left(\mathbf{x}^{t+1}\right)-\frac{\mu_{f}}{2}\left\|\mathbf{x}^*-\mathbf{x}^{t+1}\right\|^{2} \nonumber \\
&+\left\langle \mathbf{x}^*-\mathbf{x}^{t+1}, \nabla f\left(\mathbf{x}^t\right)-\nabla f\left(\mathbf{x}^{t+1}\right)\right\rangle. \nonumber
\end{align}
Substituting \eqref{a43} and \eqref{a42} into \eqref{a41}, we get
\begin{align}\label{a44}
& \frac{1}{2}\big(\|\mathbf{x}^*-\mathbf{x}^{t+1}\|_{\mathbf{Q}}^2+\|\mathbf{x}^t-\mathbf{x}^{t+1}\|_{\mathbf{Q}}^2-\|\mathbf{x}^*-\mathbf{x}^t\|_{\mathbf{Q}}^2\big) \\
\nonumber -&f(\mathbf{x}^*)+f\left(\mathbf{x}^{t+1}\right)+\frac{\mu_{f}}{2}\left\|\mathbf{x}^*-\mathbf{x}^{t+1}\right\|^{2}  \\\nonumber
+&\left\langle \mathbf{x}^*-\mathbf{x}^{t+1}, \nabla f\left(\mathbf{x}^{t+1}\right)-\nabla f\left(\mathbf{x}^t\right)+ \triangledown^2 f(\mathbf{x}^t) (\mathbf{x}^t-\mathbf{x}^{t+1})  \right\rangle\\
-&\left\langle \mathbf{x}^*-\mathbf{x}^{t+1},(\mathbf{I}-\mathbf{W})^{\frac{1}{2}} \mathbf{v}^{t+1}\right\rangle\le0. \nonumber
\end{align}
After being regrouped, \eqref{a44} becomes
\begin{align}\label{a45}
& \underbrace{f(\mathbf{x}^*)-f\left(\mathbf{x}^{t+1}\right)}_{(\romannumeral1)}+\underbrace{\left\langle \mathbf{x}^*-\mathbf{x}^{t+1},(\mathbf{I}-\mathbf{W})^{\frac{1}{2}} \mathbf{v}^{t+1}\right\rangle}_{(\romannumeral2)}\\\nonumber
&- \frac{1}{2}\big(\|\mathbf{x}^*-\mathbf{x}^{t+1}\|_{\mathbf{Q}}^2-\|\mathbf{x}^*-\mathbf{x}^t\|_{\mathbf{Q}}^2\big)\\\nonumber
\ge &\frac{1}{2}\|\mathbf{x}^t-\mathbf{x}^{t+1}\|_{\mathbf{Q}}^2 +\frac{\mu_{f}}{2}\left\|\mathbf{x}^*-\mathbf{x}^{t+1}\right\|^{2}\\\nonumber
+&\left\langle \mathbf{x}^*-\mathbf{x}^{t+1}, \nabla f\left(\mathbf{x}^{t+1}\right)-\nabla f\left(\mathbf{x}^t\right)+ \triangledown^2 f(\mathbf{x}^t) (\mathbf{x}^t-\mathbf{x}^{t+1})  \right\rangle. \nonumber
\end{align}

\textbf{Step 2.} We proceed to simplify \eqref{a45}. According to the dual update \eqref{6}, $\mathbf{v}^{t+1}=\mathbf{v}^t+\alpha(\mathbf{I} -\mathbf{W})^{\frac{1}{2}} \mathbf{x}^{t+1}$ and consequently
\begin{align} %\label{a46}
  & \big\langle {\mathbf{v}^*}-{\mathbf{v}^{t+1}},-(\mathbf{I}-\mathbf{W})^\frac{1}{2}\mathbf{x}^{t+1}\big \rangle \nonumber \\
= & \left\langle \mathbf{v}^*-\mathbf{v}^{t+1},\frac{\mathbf{v}^t-\mathbf{v}^{t+1}}{\alpha} \right\rangle \nonumber \\
= & \frac{1}{2\alpha}\left(\|\mathbf{v}^{t+1}-\mathbf{v}^t\|^2-\|\mathbf{v}^*-\mathbf{v}^t\|^2+\|\mathbf{v}^*-\mathbf{v}^{t+1}\|^2\right). \nonumber
\end{align}
Reorganizing the terms, we have
\begin{align}\label{i50}
%&\big\langle {\mathbf{v}^*}-{\mathbf{v}^{t+1}},-(\mathbf{I}-\mathbf{W})^\frac{1}{2}\mathbf{x}^{t+1}\big \rangle\nonumber\\
%&+\frac{1}{2\alpha}\left(\|\mathbf{v}^*-\mathbf{v}^t\|^2-\|\mathbf{v}^*-\mathbf{v}^{t+1}\|^2\right)\nonumber\\
&\underbrace{\big\langle {\mathbf{v}^*},-(\mathbf{I}-\mathbf{W})^\frac{1}{2}\mathbf{x}^{t+1}\big \rangle}_{(\romannumeral1')}+\underbrace{\big\langle {\mathbf{v}^{t+1}},(\mathbf{I}-\mathbf{W})^\frac{1}{2}\mathbf{x}^{t+1}\big \rangle}_{(\romannumeral2')} \\
&+\frac{1}{2\alpha}\left(\|\mathbf{v}^*-\mathbf{v}^t\|^2-\|\mathbf{v}^*-\mathbf{v}^{t+1}\|^2\right)\nonumber\\
=&\frac{1}{2\alpha}\|\mathbf{v}^{t+1}-\mathbf{v}^t\|^2. \nonumber
\end{align}

Next, we sum up \eqref{a45} and \eqref{i50}. The summation of $(\romannumeral1)$ and $(\romannumeral1')$ can be simplified as
\begin{align}\label{a47}
&f(\mathbf{x}^*)-f(\mathbf{x}^{t+1})+\big\langle \mathbf{v}^*,-(\mathbf{I}-\mathbf{W})^\frac{1}{2}\mathbf{x}^{t+1}\big \rangle  \\
=&\hat{L}(\mathbf{x}^*,\mathbf{v}^*)-\hat{L}(\mathbf{x}^{t+1},\mathbf{v}^*) \leq 0, \nonumber
\end{align}
where $\hat{L}(\mathbf{x},\mathbf{v})=f(\mathbf{x})+\langle \mathbf{v},(\mathbf{I} -\mathbf{W})^{\frac{1}{2}}\mathbf{x} \rangle$ is the Lagrangian of \eqref{c5}. The inequality holds because $(\mathbf{x}^*, \mathbf{v}^*)$ is the saddle point of $\hat{L}(\mathbf{x},\mathbf{v})$. The summation of $(\romannumeral2)$ and $(\romannumeral2')$ is
\begin{align}\label{a48}
&\big\langle \mathbf{x}^*-\mathbf{x}^{t+1},(\mathbf{I}-\mathbf{W})^\frac{1}{2}\mathbf{v}^{t+1}\big\rangle+\big\langle \mathbf{v}^{t+1},(\mathbf{I}-\mathbf{W})^\frac{1}{2}\mathbf{x}^{t+1}\big\rangle \nonumber \\
= & \big\langle \mathbf{x}^*,(\mathbf{I}-\mathbf{W})^\frac{1}{2}\mathbf{v}^{t+1}\big\rangle = 0.
\end{align}
Note that in deriving both \eqref{a47} and \eqref{a48}, we utilize the consensus condition $(\mathbf{I}-\mathbf{W})^\frac{1}{2}\mathbf{x}^*=\mathbf{0}$. With \eqref{a47} and \eqref{a48}, the summation of \eqref{a45} and \eqref{i50} is
\begin{align}\label{a51}
&\frac{1}{2}\left(\|\mathbf{x}^*\!-\!\mathbf{x}^t\|_{\mathbf{Q}}^2-\|\mathbf{x}^*\!-\!\mathbf{x}^{t+1}\|_{\mathbf{Q}}^2\right)\\\nonumber&+\frac{1}{2\alpha}\left(\|\mathbf{v}^*-\mathbf{v}^t\|^2-\|\mathbf{v}^*-\mathbf{v}^{t+1}\|^2\right)\\\nonumber
\ge&\frac{1}{2}\|\mathbf{x}^t\!-\!\mathbf{x}^{t+1}\|_{\mathbf{Q}}^2+\frac{1}{2\alpha}\|\mathbf{v}^{t+1}-\mathbf{v}^t\|^2+\frac{\mu_{f}}{2}\left\|\mathbf{x}^*-\mathbf{x}^{t+1}\right\|^{2}\\\nonumber
+&\left\langle \mathbf{x}^*-\mathbf{x}^{t+1}, \nabla f\left(\mathbf{x}^{t+1}\right)-\nabla f\left(\mathbf{x}^t\right)+ \triangledown^2 f(\mathbf{x}^t) (\mathbf{x}^t-\mathbf{x}^{t+1})  \right\rangle. \nonumber
\end{align}

It is the $\mu_f$-strong convexity of $f$ that brings the quadratic term $\frac{\mu_f}{2}\|\mathbf{x}^*-\mathbf{x}^{t+1}\|^2$ in \eqref{a51}, which enables us to establish the linear convergence. Indeed, by Cauchy-Schwarz inequality, for any $\theta > 0$ we have
\begin{align}\label{a52-tttt}
&\left\langle \mathbf{x}^*-\mathbf{x}^{t+1}, \nabla f\left(\mathbf{x}^{t+1}\right)-\nabla f\left(\mathbf{x}^t\right)+ \triangledown^2 f(\mathbf{x}^t) (\mathbf{x}^t-\mathbf{x}^{t+1})  \right\rangle \nonumber\\
&\geq -\frac{1}{\theta}\|\nabla f\left(\mathbf{x}^{t+1}\right)-\nabla f\left(\mathbf{x}^t\right)+ \triangledown^2 f(\mathbf{x}^t) (\mathbf{x}^t-\mathbf{x}^{t+1})\|^2 \nonumber\\
&\quad -\theta\|\mathbf{x}^*-\mathbf{x}^{t+1}\|^2.
\end{align}	
By Lipschitz continuity of $\nabla f$, it holds
\begin{align}\label{fadfasdfsdddd}
&-\frac{1}{\theta}\|\nabla f\left(\mathbf{x}^{t+1}\right)-\nabla f\left(\mathbf{x}^t\right)+ \triangledown^2 f(\mathbf{x}^t) (\mathbf{x}^t-\mathbf{x}^{t+1})\|^2 \nonumber\\
\geq &-\frac{2}{\theta}\|\nabla f\left(\mathbf{x}^{t+1}\right)-\nabla f\left(\mathbf{x}^t\right)\|^2-\frac{2}{\theta}\|\triangledown^2 f(\mathbf{x}^t) (\mathbf{x}^t-\mathbf{x}^{t+1})\|^2 \nonumber \\
\geq & - \frac{4L_f^2}{\theta}\|\mathbf{x}^t-\mathbf{x}^{t+1}\|^2.
\end{align}
Thus, combining \eqref{a52-tttt} and \eqref{fadfasdfsdddd} yields
\begin{align}\label{a52}
&\left\langle \mathbf{x}^*-\mathbf{x}^{t+1}, \nabla f\left(\mathbf{x}^{t+1}\right)-\nabla f\left(\mathbf{x}^t\right)+ \triangledown^2 f(\mathbf{x}^t) (\mathbf{x}^t-\mathbf{x}^{t+1})  \right\rangle \nonumber\\
&\geq -\theta\|\mathbf{x}^*-\mathbf{x}^{t+1}\|^2- \frac{4L_f^2}{\theta}\|\mathbf{x}^t-\mathbf{x}^{t+1}\|^2.
\end{align}	
substituting \eqref{a52} into \eqref{a51}, we obtain
\begin{align}\label{a53}
&\|\mathbf{x}^*\!-\!\mathbf{x}^t\|_{\mathbf{Q}}^2-\|\mathbf{x}^*\!-\!\mathbf{x}^{t+1}\|_{\mathbf{Q}}^2\\\nonumber&+\frac{1}{\alpha}\left(\|\mathbf{v}^*-\mathbf{v}^t\|^2-\|\mathbf{v}^*-\mathbf{v}^{t+1}\|^2\right)\\\nonumber
\ge&\|\mathbf{x}^t\!-\!\mathbf{x}^{t+1}\|_{\mathbf{Q}}^2+\frac{1}{\alpha}\|\mathbf{v}^{t+1}-\mathbf{v}^t\|^2+{\mu_{f}}\left\|\mathbf{x}^*-\mathbf{x}^{t+1}\right\|^{2}\\\nonumber
&-\theta\|\mathbf{x}^*-\mathbf{x}^{t+1}\|^2-\frac{4L_f^2}{\theta}\|\mathbf{x}^t-\mathbf{x}^{t+1}\|^2\\\nonumber
=&\|\mathbf{x}^t\!-\!\mathbf{x}^{t+1}\|_{({\mathbf{Q}}-\frac{4L_f^2}{\theta}\mathbf{I})}^2+\frac{1}{\alpha}\|\mathbf{v}^{t+1}-\mathbf{v}^t\|^2\\\nonumber
&+({\mu_{f}-\theta})\left\|\mathbf{x}^*-\mathbf{x}^{t+1}\right\|^{2}. \nonumber
\end{align}

\textbf{Step 3.} To prove the linear convergence, the parameters in \eqref{a53} are required to satisfy
\begin{equation}\label{a54}
\left\{\begin{aligned}
&{\lambda_{\min }({\mathbf{Q}})}-\frac{4L_f^2}{\theta}>  0, \\
& \mu_f-\theta>0. \\
\end{aligned}\right.
\end{equation}
Hence, \eqref{a54} is attainable when
\begin{equation} \label{a55}
\delta\triangleq 1-\frac{4L_f^2}{\mu_f {\lambda_{\min }({\mathbf{Q}})} } > 0,
\end{equation}
%Because $\mathbf{Q}=\epsilon \mathbf{I}  -\alpha(\mathbf{I}-\mathbf{W})$, ${\lambda_{\min }({\mathbf{Q}})} \geq \epsilon-\alpha \lambda_{\max}(\mathbf{I}-\mathbf{W})$.
which holds since ${\lambda_{\min }({\mathbf{Q}})}  = \epsilon-\alpha \lambda_{\max}(\mathbf{I}-\mathbf{W}) > \frac{4L_f^2}{\mu_f}$ by hypothesis.

When $\delta>0$, then \eqref{a54} holds true if we choose $\theta=\frac{\mu_f}{1+\delta}$. Substituting this specific $\theta$ and the definition of $\delta$, we can rewrite \eqref{a53} to
\begin{align}\label{a56}
&\|\mathbf{x}^*\!-\!\mathbf{x}^t\|_{\mathbf{Q}}^2-\|\mathbf{x}^*\!-\!\mathbf{x}^{t+1}\|_{\mathbf{Q}}^2 \\\nonumber&+\frac{1}{\alpha}\left(\|\mathbf{v}^*-\mathbf{v}^t\|^2-\|\mathbf{v}^*-\mathbf{v}^{t+1}\|^2\right)\\\nonumber
\ge& \delta^2 {\lambda_{\min }({\mathbf{Q}})}   \|\mathbf{x}^t\!-\!\mathbf{x}^{t+1}\|^2+\frac{1}{\alpha}\|\mathbf{v}^{t+1}-\mathbf{v}^t\|^2\\\nonumber
&+\frac{\mu_f \delta}{1+\delta}\left\|\mathbf{x}^*-\mathbf{x}^{t+1}\right\|^{2}. \nonumber
\end{align}

To establish the linear convergence in \eqref{26}, we need to show that $ \left\|\mathbf{\zeta}^t-\mathbf{\zeta}^*\right\|_{\mathbf{G}}^{2}- \left\|\mathbf{\zeta}^{t+1}-\mathbf{\zeta}^*\right\|_{\mathbf{G}}^{2}\ge \delta^{\prime}\left\|\mathbf{\zeta}^{t+1}-\mathbf{\zeta}^*\right\|_{\mathbf{G}}^{2}$.
Given \eqref{a56}, it is enough to show that
\begin{align}\label{a57}
&\frac{ \delta^{\prime}}{\alpha}\left\|\mathbf{v}^{t+1}-\mathbf{v}^*\right\|^{2}+ \delta^{\prime} \left\|\mathbf{x}^{t+1}-\mathbf{x}^*\right\|_{\mathbf{Q}}^{2} \\ \nonumber
\leq&\delta^2 {\lambda_{\min }({\mathbf{Q}})} \|\mathbf{x}^t\!-\!\mathbf{x}^{t+1}\|^2+\frac{1}{\alpha}\|\mathbf{v}^{t+1}-\mathbf{v}^t\|^2 \\ \nonumber
& +\frac{\mu_f \delta}{1+\delta}\left\|\mathbf{x}^*-\mathbf{x}^{t+1}\right\|^{2}.
\end{align}

We proceed to find an upper bound for $\left\|\mathbf{v}^{t+1}-\mathbf{v}^*\right\|^{2}$ in terms of the summands at the right-hand side
of \eqref{a57}. For $\nabla f\left(\mathbf{x}^{t+1}\right) -\nabla f\left(\mathbf{x}^*\right)+(\mathbf{I}-\mathbf{W})^{\frac{1}{2}}\left(\mathbf{v}^{t+1}-\mathbf{v}^*\right)
+\epsilon\left(\mathbf{x}^{t+1}-\mathbf{x}^t\right)+\mathbf{e}^t=\mathbf{0}$ in \eqref{7}, we utilize Cauchy-Schwarz inequality twice to obtain
\begin{align}\label{a58-tttt}
&\left\|\mathbf{v}^{t+1}-\mathbf{v}^*\right\|_{\mathbf{I}-\mathbf{W}}^{2} \leq \frac{{\beta} \epsilon^2}{\beta-1}\|\mathbf{x}^{t+1}-\mathbf{x}^t\|^2 \\\nonumber
& \hspace{4em} +\beta\phi\| \nabla f\left(\mathbf{x}^{t+1}\right) -\nabla f\left(\mathbf{x}^*\right) \|^2+\frac{\beta {\phi}}{\phi-1}\|\mathbf{e}^t\|^2,\nonumber
\end{align}
where $\beta>1$ and $\phi>1$ are parameters introduced in using Cauchy-Schwarz inequality. By Lipschitz continuity of $\nabla f$, it holds that $\| \nabla f\left(\mathbf{x}^{t+1}\right) -\nabla f\left(\mathbf{x}^*\right) \|^2 \leq L_f^2 \| \mathbf{x}^{t+1} - \mathbf{x}^* \|^2$. By \eqref{25}, we have $\|\mathbf{e}^t\|^2 \leq \kappa^2 \left\|\mathbf{x}^{t+1}-\mathbf{x}^t\right\|^2$. Therefore, \eqref{a58-tttt} implies that
\begin{align*}%\label{a58-tttt-aaaa}
&\left\|\mathbf{v}^{t+1}-\mathbf{v}^*\right\|_{\mathbf{I}-\mathbf{W}}^{2} \\
\leq & \left( \frac{{\beta} \epsilon^2}{\beta-1} + \frac{\beta \phi \kappa^2}{\phi-1} \right) \|\mathbf{x}^{t+1}-\mathbf{x}^t\|^2 + \beta\phi L_f^2\| \mathbf{x}^{t+1} - \mathbf{x}^* \|^2. \nonumber
\end{align*}
Further, considering that $\mathbf{v}^{t+1}$ and $\mathbf{v}^*$ both lie in the column space of $(\mathbf{I}-\mathbf{W})^{\frac{1}{2}}$, we have
\begin{align}\label{a58}
&\left\|\mathbf{v}^{t+1}-\mathbf{v}^*\right\|^{2} \leq \frac{1}{\hat{\lambda}_{\min}(\mathbf{I}-\mathbf{W})} \\\nonumber
& \left\{ \left( \frac{{\beta} \epsilon^2}{\beta-1} + \frac{\beta \phi \kappa^2}{\phi-1} \right) \|\mathbf{x}^{t+1}-\mathbf{x}^t\|^2 + \beta\phi L_f^2\| \mathbf{x}^{t+1} - \mathbf{x}^* \|^2 \right\}.\nonumber
\end{align}
Note that $\hat{\lambda}_{\min}(\mathbf{I}-\mathbf{W})>0$ because $ \mathbf{I}-\mathbf{W} \succeq 0$. We also find an upper bound for $\left\|\mathbf{x}^{t+1}-\mathbf{x}^*\right\|_{\mathbf{Q}}^{2}$ as
\begin{align}\label{faeftrewes}
\left\|\mathbf{x}^{t+1}-\mathbf{x}^*\right\|_{\mathbf{Q}}^{2} \leq \lambda_{\max}(\mathbf{Q}) \left\|\mathbf{x}^{t+1}-\mathbf{x}^*\right\|^{2}.
\end{align}

By substituting the upper bounds in \eqref{a58} and \eqref{faeftrewes} into \eqref{a57}, we obtain a sufficient condition for \eqref{26}, given by
\begin{align}\label{a59}
&\lambda_{\max}(\mathbf{Q}) {\delta}^{\prime} \left\|\mathbf{x}^{t+1}-\mathbf{x}^*\right\|^{2} + \frac{{\delta}^{\prime}}{\alpha\hat{\lambda}_{\min}(\mathbf{I}-\mathbf{W})}\!\nonumber\\\nonumber&\left\{\! \left( \frac{{\beta} \epsilon^2}{\beta-1} + \frac{\beta \phi \kappa^2}{\phi-1} \right) \|\mathbf{x}^{t+1}-\mathbf{x}^t\|^2 + \beta\phi L_f^2\| \mathbf{x}^{t+1} - \mathbf{x}^* \|^2\right\}\! \\
\leq& \delta^2{\lambda_{\min }({\mathbf{Q}})} \|\mathbf{x}^t\!-\!\mathbf{x}^{t+1}\|^2+\frac{1}{\alpha}\|\mathbf{v}^{t+1}-\mathbf{v}^t\|^2 \nonumber \\
&+\frac{\mu_f \delta}{1+\delta}\left\|\mathbf{x}^*-\mathbf{x}^{t+1}\right\|^{2}.
\end{align}
Regrouping the terms, we know that \eqref{a59} is equivalent to
\begin{align}\label{a60}
&\left(\frac{\mu_f\delta}{1+\delta}-{\delta }^{\prime}\lambda_{\max}({\mathbf{Q}})-\frac{{\delta }^{\prime}\beta\phi L_f^2}{\alpha \hat{\lambda}_{\min}(\mathbf{I}-\mathbf{W})}\right)\left\|\mathbf{x}^{t+1}-\mathbf{x}^*\right\|^2 \nonumber \\\nonumber
&+\left(\delta^2\lambda_{\min}(\mathbf{Q})-\frac{{\delta}^{\prime }\beta\epsilon^2/(\beta-1)}{{\alpha \hat{\lambda}_{\min}(\mathbf{I}-\mathbf{W})}} -\frac{{\delta}^{\prime}\beta\phi \kappa^2/(\phi -1)}{{\alpha \hat{\lambda}_{\min}(\mathbf{I}-\mathbf{W})}}\right) \\
& \quad \left\|\mathbf{x}^{t+1}-\mathbf{x}^t\right\|^{2}
+\frac{1}{\alpha}\left\|\mathbf{v}^{t+1}-\mathbf{v}^t\right\|^{2}\ge 0.
\end{align}

Recall that if \eqref{a60} is satisfied, then \eqref{a59} holds, and hence \eqref{a57} and \eqref{26} are also true. To get \eqref{a60}, we need to make sure that the coefficients in \eqref{a60} are non-negative. Thus, \eqref{a60} holds if ${\delta}^{\prime}$ satisfies
\begin{align}\label{d65}
{\delta}^{\prime} \leq \min\left\{\frac{\mu_f\delta}{(1+\delta) \left[\lambda_{\max}(\mathbf{Q})+\frac{\beta\phi L_f^2}{{\alpha \hat{\lambda}_{\min}(\mathbf{I}-\mathbf{W})}} \right]},\right.\\\left.\frac{\alpha\delta^2 \lambda_{\min}(\mathbf{Q}) { \hat{\lambda}_{\min}(\mathbf{I}-\mathbf{W})}}{\frac{\beta\epsilon^2}{(\beta-1)}+\frac{\beta\phi\kappa^2}{(\phi-1)}}\right\}.\nonumber
\end{align}

By the definition of $\mathbf{Q}=\epsilon \mathbf{I} -\alpha(\mathbf{I}-\mathbf{W})$, we have
\begin{align}
&{\lambda_{\min }({\mathbf{Q}})} = \epsilon-\alpha \lambda_{\max}(\mathbf{I}-\mathbf{W}) > \frac{4L_f^2}{\mu_f}>0, \nonumber \\
&{\lambda_{\max }({\mathbf{Q}})} = \epsilon-\alpha \lambda_{\min}(\mathbf{I}-\mathbf{W}) = \epsilon >0. \nonumber
\end{align}
Substituting these connections and the definition of $\kappa=2L_f$ $+\alpha \lambda_{\max}(\mathbf{I}-\mathbf{W})$ to \eqref{d65}, we eventually find the largest ${\delta}^{\prime}$ to satisfy \eqref{d65}, as in \eqref{xx}.
\end{proof}

\autoref{theom1} establishes the linear convergence of sequence $\left\|\mathbf{\zeta}^t-\mathbf{\zeta}^*\right\|_{\mathbf{G}}^{2}$, where the factor
of linear convergence is $\frac{1}{1 + {\delta}^{\prime}}$. When $\lambda_{\max}(\mathbf{I}-\mathbf{W})$ increases, $\delta$ monotonically decreases and $\delta'$ monotonically decreases. On the other hand, when $\hat{\lambda}_{\min}(\mathbf{I}-\mathbf{W})$ increases, $\delta'$ also monotonically increase. These observations indicate the impact of network topology on the convergence speed. Since $\mathbf{Q}$ is positive definite under the parameter setting, $\mathbf{G}$ is also positive definite such that $\mathbf{x}^t$ converges linearly to $\mathbf{x}^*$.

\section{Numerical Experiments}
We consider the application of Newton tracking for solving a decentralized logistic regression problem in the form of
\begin{align}\nonumber
{{x}}^{*} = \underset{{x} \in \mathbb{R}^{p}}{\operatorname{argmin}} ~ \frac{\rho}{2}\|{x}\|^{2}+\sum_{i=1}^{n} \sum_{j=1}^{m_{i}} \ln \left(1+\exp \left(-\left(\mathbf{o}_{i j}^{T} {x}\right) \mathbf{p}_{i j}\right)\right),
\end{align}
where node $i$ has access to $m_i$ training samples $(\mathbf{o}_{i j}, \mathbf{p}_{i j}) \in \mathbb{R}^{p} \times\{-1,+1\}$; $j=1, \ldots, m_{i}$. We add a regularization term $\frac{\rho}{2}\|{x}\|^{2}$ with $\rho>0$ to the loss function for avoiding overfitting. In the numerical experiments, we randomly generate the elements in $\mathbf{o}_{i j}$ following the normal distribution and the elements in $\mathbf{p}_{i j}$ following the uniform distribution. We randomly generate $\frac{\tau n(n-1)}{2}$ undirected edges for the network of $n$ nodes, where $\tau \in(0,1] $ is the connectivity ratio, while guarantee that the network is connected.

To evaluate performance of the compared algorithms, the optimal logistic classifier $x^*$ is pre-computed through centralized gradient descent. The performance metric is relative error, defined as $\left\|\mathbf{x}^{t}-\mathbf{x}^{*}\right\| /\left\|\mathbf{x}^{0}-\mathbf{x}^{*}\right\|$. We conducted the experiments with Matlab R2016b, running on a laptop with Intel(R) Core(TM) i7 CPU@1.80GHz, 16.0 GB of RAM, and Windows 10 operating system.

\subsection{Comparison with Second-order Methods}\label{exp1}
We compare Newton tracking with second-order algorithms including NN-$K$ \cite{mokhtari2016network}, ESOM-$K$ \cite{mokhtari2016decentralized} and DQM \cite{mokhtari2016dqm}. In every iteration of NN-$K$ and ESOM-$K$, the nodes need to execute a $K+1$-round inner loop to compute the inverse of a topology-dependent matrix in the forms of $\alpha\triangledown^2 f(\mathbf{x})+ (\mathbf{I} -\mathbf{W})$ and $\triangledown^2 f(\mathbf{x})+\alpha (\mathbf{I} -\mathbf{W})+\epsilon \mathbf{I}$, respectively.

%In every iteration of DQM, each node $i$ needs to solve an optimization problem whose complexity is determined by $f_i$, such that the computational cost of DQM is higher than the others'.

In the first experiment, we set the number of nodes as $n = 10$ and the connectivity ratio as  $\tau=0.5$. Each node holds $12$ samples, i.e., $m_i = 12$, for all $i$. The dimension of sample vectors $\mathbf{o}_{ij}$ is $p =8$. We set the regularization parameter $\rho=0.001$.

We run NN-$K$, ESOM-$K$, and DQM with fixed hand-optimized step sizes. The step sizes of DQM is $\alpha=0.3$. The parameters of ESOM-0, ESOM-1 and ESOM-2 are $\alpha=3.3$ and $\epsilon=3$. For NN-$K$, a smaller step size improves accuracy but leads to slow convergence, while a larger step size accelerates the convergence at the cost of low accuracy. Therefore, for NN-0, NN-1 and NN-2 we set
$\alpha = 0.001$, $\alpha = 0.008$, and $\alpha = 0.02$, respectively. For Newton tracking, we set the parameters the same as ESOM, i.e., $\alpha=3.3$ and $\epsilon=3$.
\begin{figure}
	\centering
	\centerline{\includegraphics[width=8cm]{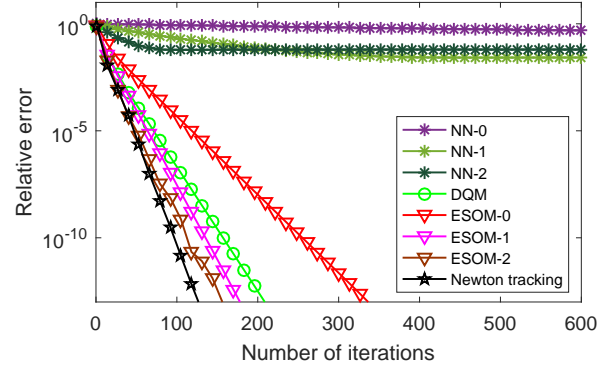}}
	\caption{Relative errors of Newton tracking,  DQM, NN-$K$, and ESOM-$K$ versus number of iterations.}
    \label{fig1}
\end{figure}

\begin{figure}
	\centering
	\centerline{\includegraphics[width=8cm]{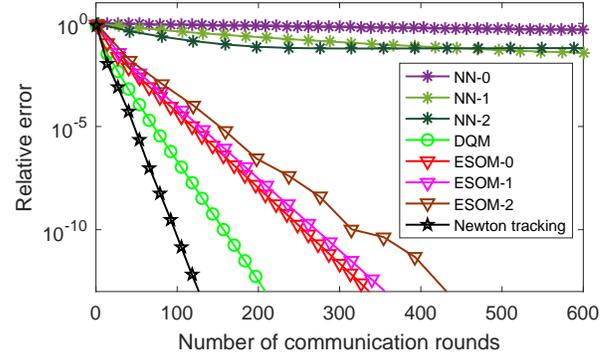}}
	\caption{Relative errors of Newton tracking,  NN-$K$, ESOM-$K$ and DQM versus rounds of communications.}
	\label{fig2}
\end{figure}

Fig. \ref{fig1} illustrates the relative error versus the number of iterations. Observe that NN-$K$ converges to the neighborhoods of optimal argument. Among the exact decentralized algorithms, the proposed Newton tracking algorithm has the best performance compared with the other algorithms and converges linearly, which validates the theoretical result in \autoref{theom1}.

Newton tracking and DQM require one round of communication per iteration. For other algorithms, NN-$K$ and ESOM-$K$ require $K + 1$ rounds. Fig. \ref{fig2} illustrates the relative error versus the rounds of communication. Observe that although ESOM-1 and ESOM-2 perform well as depicted in Fig. \ref{fig1}, they become worse in Fig. \ref{fig2}  because more rounds of communication are required in each iteration. In terms of the communication cost, the proposed Newton tracking algorithm is still the best.

\subsection{Comparison with First-order Methods}

\begin{figure}[t]
	\centering
	
	\subfigure{
		\includegraphics[width=8cm] {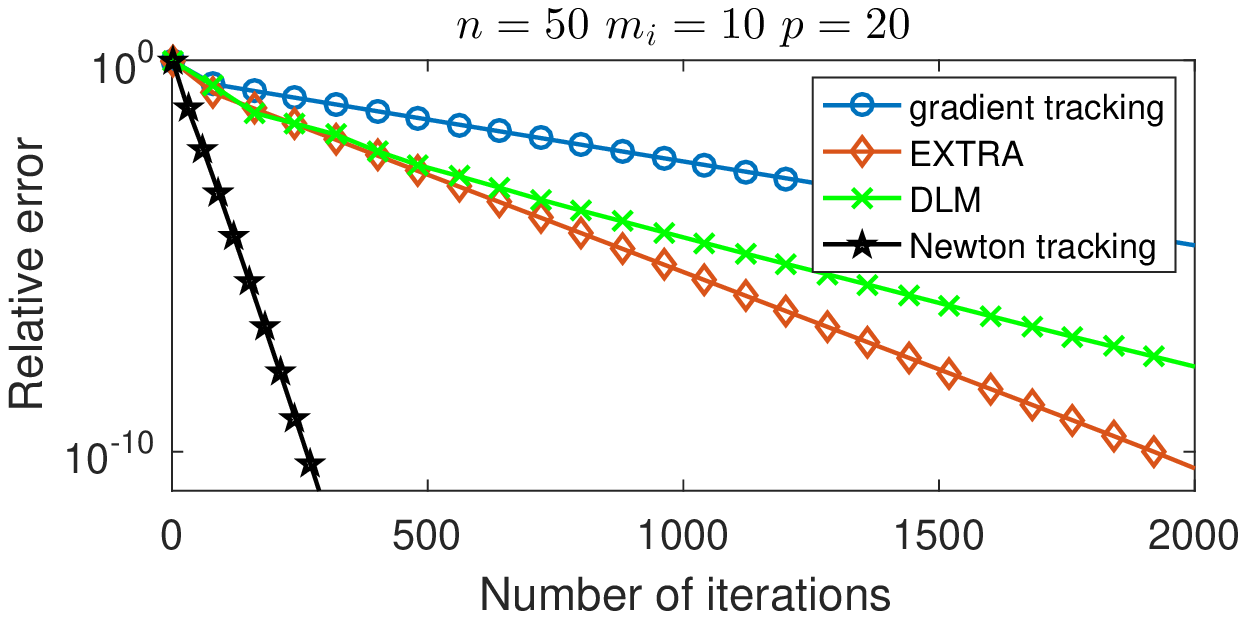}
		\label{fig:subfig1}
	}
	
	\subfigure{
		\includegraphics[width=8cm] {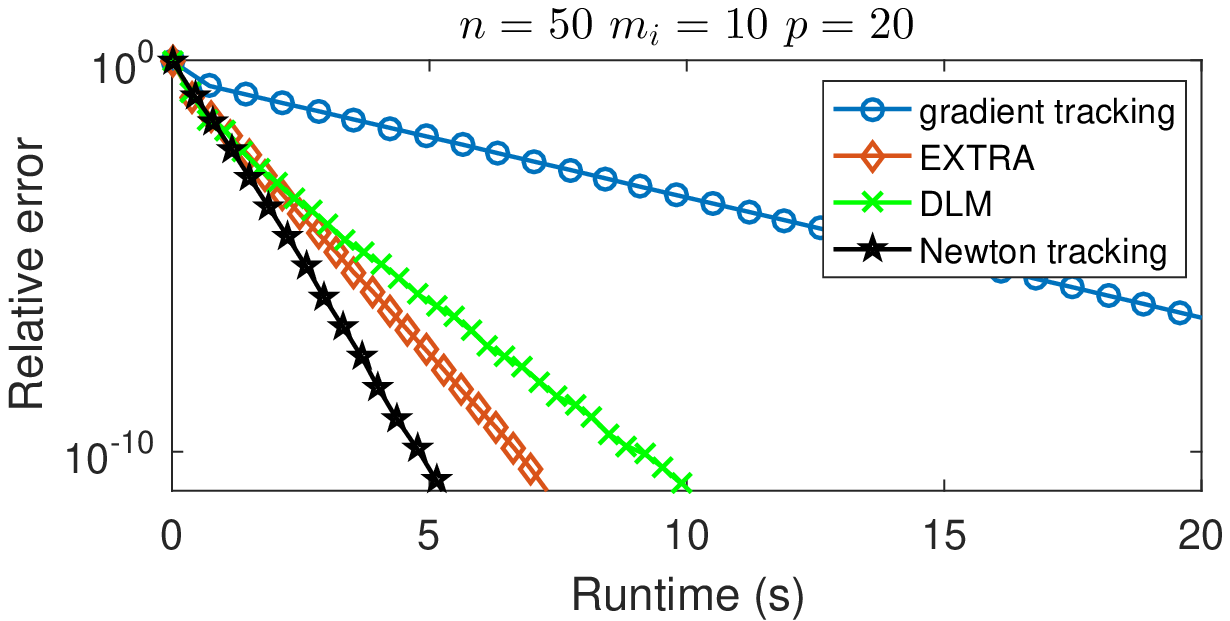}
		\label{fig:subfig2}
	}
		\caption{Relative errors of Newton tracking, gradient tracking, EXTRA and DLM when $n=50$, $m_i=10$ and $p=20$.}
		\label{fig50}
\end{figure}

\begin{figure}	
\centering
		
	\subfigure{
		\includegraphics[width=8cm] {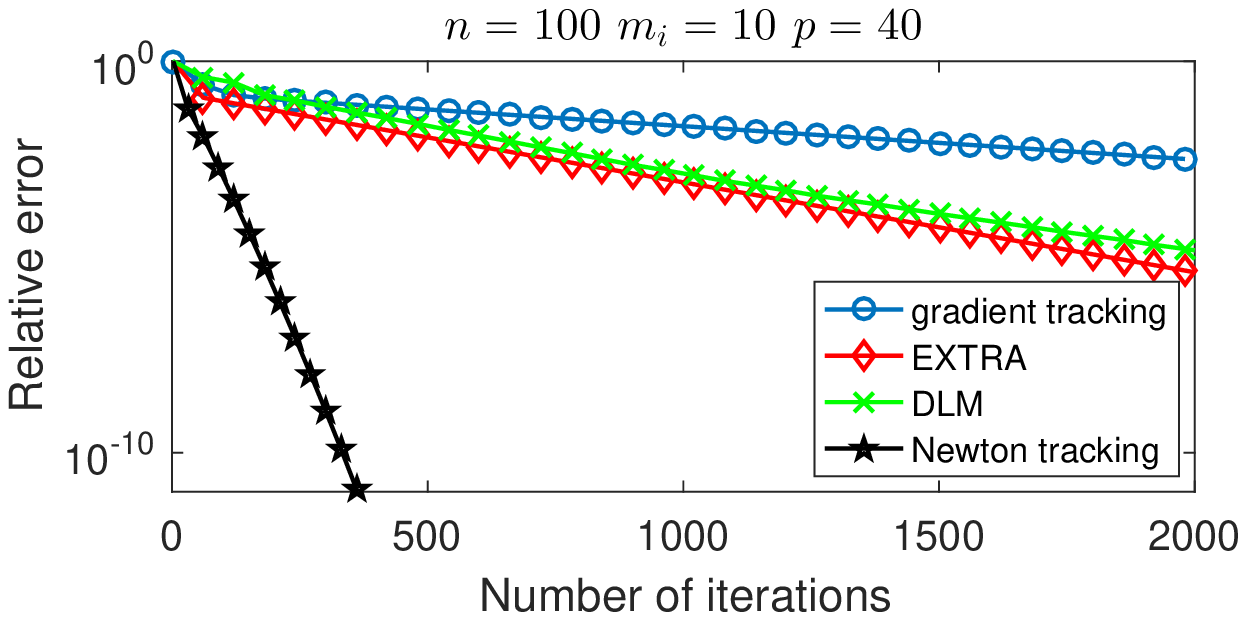}
		\label{fig:subfig3}
	}
	
		\subfigure{
		\includegraphics[width=8cm] {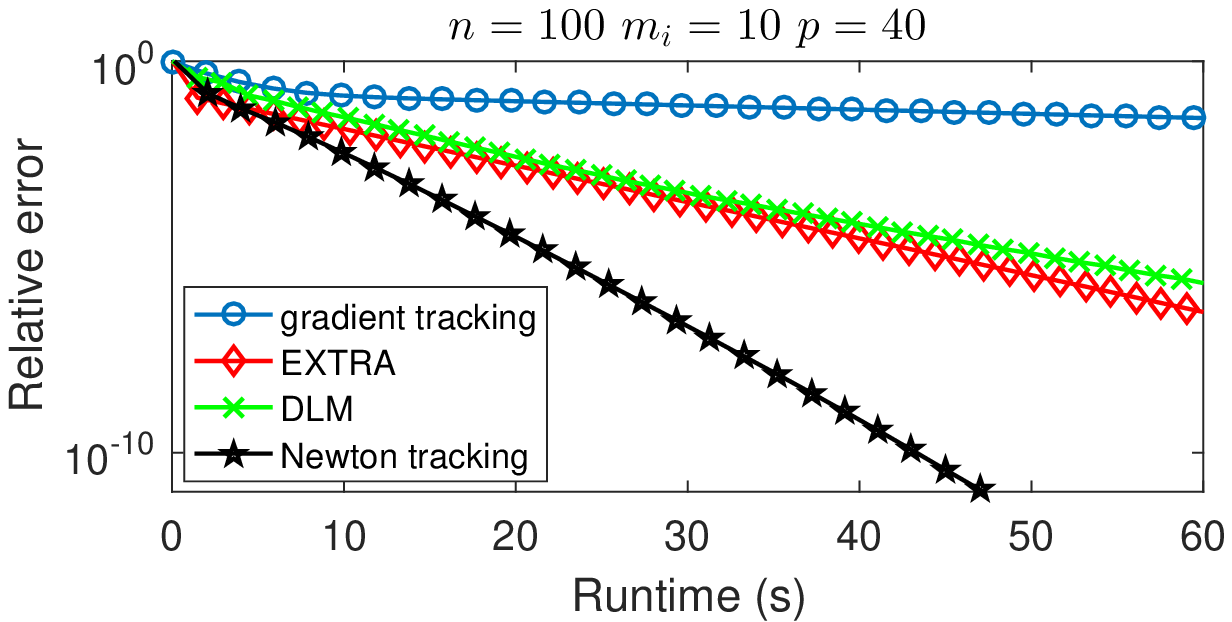}
		\label{fig:subfig3}
	}
	\caption{Relative errors of Newton tracking, gradient tracking, EXTRA and DLM when $n=100$, $m_i=10$ and $p=40$.}
	\label{fig100}
\end{figure}
We compare Newton tracking with the first-order  algorithms, including gradient tracking \cite{qu2017harnessing}, EXTRA \cite{shi2015extra} and DLM \cite{ling2015dlm}. We respectively rewrite these three algorithms as their equivalent updates:  (gradient tracking) $\mathbf{x}^{t+2}=2\mathbf{W} \mathbf{x}^{t+1}-\mathbf{W}^{2}\mathbf{x}^{t}-\alpha\left[\nabla f\left(\mathbf{x}^{t+1}\right)-\nabla f\left(\mathbf{x}^{t}\right)\right] $, (EXTRA) $
\mathbf{x}^{t+2}=(\mathbf{I}+\mathbf{W}) \mathbf{x}^{t+1}-(\mathbf{I}+\mathbf{W})\mathbf{x}^{t}/2-\alpha\left[\nabla f\left(\mathbf{x}^{t+1}\right)-\nabla f\left(\mathbf{x}^{t}\right)\right]
$
and (DLM) $
\mathbf{x}^{t+2}=\left(\mathbf{I}-\alpha DL_o\right)( 2\mathbf{x}^{t+1}-\mathbf{x}^{t})- D\left[\nabla f\left(\mathbf{x}^{t+1}\right)-\nabla f\left(\mathbf{x}^{t}\right)\right]
$, where $D=\text{diag}\{1/(2\alpha d_i + \epsilon)\}$, and $d_i$ is the degree of node $i$. $L_o$ is the oriented Laplacian defined in \cite{ling2015dlm}.  

In this section we conduct the experiments under two larger networks. In the second (third) experiments we set the number of nodes as $n = 50 \; (100)$, the number of samples on each agent as $m_i=10 \;(10)$ for all $i$ and the dimension of sample vectors as $p =20\; (40)$. The other settings are the same as \autoref{exp1}. 

We run all the algorithms with fixed hand-optimized step sizes. The step sizes of gradient tracking and EXTRA in the second (third) experiments are $\alpha=0.16\; (0.6)$ and $\alpha=0.07 \;(1.6)$, respectively. The parameters of DLM in the second (third) experiments are $\alpha=0.1 \;(0.008)$ and $\epsilon=0.1 \;(0.001)$.
For Newton tracking, the parameters in the second (third) experiments are $\alpha=1.1 \;(0.08)$ and $\epsilon=1.2 \;(0.08)$.

Fig. \ref{fig50} illustrates the relative error versus the number of iterations and runtime, respectively. Observe that the proposed Newton tracking outperforms the first-order algorithms in terms of either the number of iterations or runtime. Although Newton tracking computes the inverse of estimated Hessian $\nabla^2 f_i (x_i) + \epsilon I_p \in \mathbb{R}^{p\times p}$ in each iteration, it calls for relatively smaller number of iterations compared with the first-order algorithms, which leads to the shorter runtime. We get similar results in the third experiments, see Fig. \ref{fig100}. 

\subsection{Effect of Network Topology}
This section investigates the performance of Netwon tracking in four different topologies  including line graph, cycle graph, random graphs with $\tau =\{0.3, 0.5,0.7\}$, and complete graph. The parameters of Newton tracking are set as $\alpha=2.3$ and $\epsilon=2.4$. All the other settings are the same as those in \autoref{exp1}.

\begin{figure}
	\centering
	\centerline{\includegraphics[width=8cm]{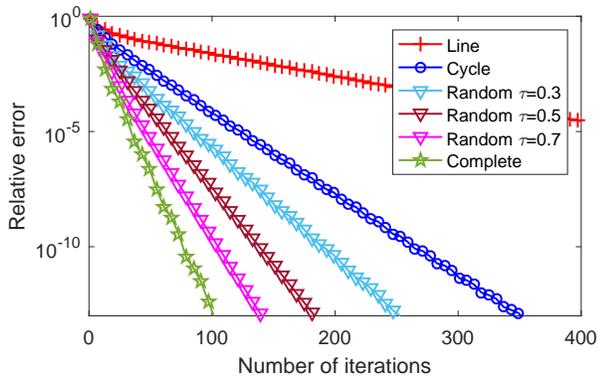}}
	\caption{Relative errors of Newton tracking versus number of iterations for line graph, cycle graph, random graphs with $\tau =\{0.3, 0.5,0.7\}$, and complete graph.}
	\label{fig3}
\end{figure}

Fig. \ref{fig3} illustrates the relative errors versus the number of iterations. Observe that the proposed Newton tracking algorithm has linear convergence rates in all types of graphs. Among them, complete graph yields the fastest speed. This observation confirms the convergence rate developed in \autoref{theom1}. To be specific, for line graph, cycle graph, random graphs with $\tau=\{0.3, 0.5, 0.7\}$, and complete graph, we have $\hat{\lambda}_{\min }(\mathbf{I}-\mathbf{W})=\{0.03,0.12,0.17,0.34, 0.43,1.00\}$ and ${\lambda}_{\max}(\mathbf{I}-\mathbf{W})=\{1.30,1.33,1.16,1.15,1.10,1.00 \}$, respectively. According to the definition of $\delta^{\prime}$ in \eqref{d65}, the complete graph with the largest $\hat{\lambda}_{\min }(\mathbf{I}-\mathbf{W})$ and the smallest ${\lambda}_{\max }(\mathbf{I}-\mathbf{W})$ has the largest $\delta^{\prime}$, and hence the fastest convergence speed.

%\begin{table}
%\caption{An Example of a Table}
%\label{table_example}
%\begin{center}
%\begin{tabular}{|c||c|}
%	\hline
%	One & Two\\
%	\hline
%	Three & Four\\
%	\hline
%\end{tabular}
%\end{center}
%\end{table}

%\begin{figure}[thpb]
%\centering
%%\includegraphics[scale=1.0]{figurefile}
%\caption{Inductance of oscillation winding on amorphous
%magnetic core versus DC bias magnetic field}
%\label{figurelabel}
%\end{figure}

%%%%%%%%%%%%%%%%%%%%%%%%%%%%%%%%%%%%%%%%%%%%%%%%%%%%%%%%%%%%%%%%%%%%%%%%%%%%%%%%
\section{CONCLUSIONS}
This paper proposed a novel Newton tracking algorithm to solve the decentralized consensus optimization problem. Each node updates its local variable along a modified local Newton direction, which is calculated with neighboring and historical information. Newton tracking employs a fixed step size and, yet, converges to an exact solution. The connections between Newton tracking and several existing methods, including gradient tracking and second-order algorithms were investigated. We proved that the proposed algorithm converges at a linear rate under the strongly convex assumption. Numerical experiments demonstrated the efficacy of Newton tracking, compared with existing algorithms such as gradient tracking, NN, ESOM, and DQM.

\bibliographystyle{IEEEtran}
\bibliography{IEEEabrv_PnADMM}

%\begin{thebibliography}{99}
%
%\bibitem{c1}
%J.G.F. Francis, The QR Transformation I, {\it Comput. J.}, vol. 4, 1961, pp 265-271.
%
%\bibitem{c2}
%H. Kwakernaak and R. Sivan, {\it Modern Signals and Systems}, Prentice Hall, Englewood Cliffs, NJ; 1991.
%
%\bibitem{c3}
%D. Boley and R. Maier, "A Parallel QR Algorithm for the Non-Symmetric Eigenvalue Algorithm", {\it in Third SIAM Conference on Applied Linear Algebra}, Madison, WI, 1988, pp. A20.
%
%\end{thebibliography}
\addtolength{\textheight}{-3cm}   % This command serves to balance the column lengths
\end{document}